\documentclass[12pt]{amsart}
\usepackage[active]{srcltx}
\usepackage{latexsym}
\usepackage{hyperref}

\newcommand{\ignore}[1]{}

\newcommand{\hide}[1]{}

\DeclareMathOperator{\ad}{ad}

\DeclareMathOperator{\Hom}{Hom}

\newcommand{\Z}[0]{\mathbb Z}

\newtheorem{dummy}{Dummy}

\newtheorem{lemma}[dummy]{Lemma}
\newtheorem*{lemma*}{Lemma}
\newtheorem{theorem}[dummy]{Theorem}

\newtheorem{cor}[dummy]{Corollary}

\theoremstyle{definition}

\theoremstyle{remark}
\newtheorem{rem}[dummy]{Remark}

\begin{document}

\bibliographystyle{alpha}
\author{Sandro Mattarei}
\title[A sandwich in thin Lie algebras]{A sandwich in thin Lie algebras}

\begin{abstract}
A thin Lie algebra is a Lie algebra $L$, graded over the positive integers,
with its first homogeneous component $L_1$ of dimension two and generating $L$, and
such that each nonzero ideal of $L$ lies between consecutive terms of its lower central series.
All its homogeneous components have dimension one or two,
and the two-dimensional components are called diamonds.
We prove that if the next diamond past $L_1$ of an infinite-dimensional thin Lie algebra $L$ is $L_k$, with $k>5$,
then $[Lyy]=0$ for some nonzero element $y$ of $L_1$.
\end{abstract}

\maketitle

\section{Introduction}\label{sec:intro}

A {\em thin} Lie algebra is
a graded Lie algebra $L=\bigoplus_{i=1}^{\infty}L_i$
with $\dim(L_1)=2$ and satisfying the following {\em covering property:}
for each $i$, each nonzero $z\in L_i$ satisfies $[zL_1]=L_{i+1}$.
This implies at once that homogeneous components of a thin Lie algebra are at most two-dimensional.
Those components of dimension two are called {\em diamonds,}
hence $L_1$ is a diamond, and if there are no other diamonds then $L$ is a graded Lie algebra of maximal class, see~\cite{CMN,CN}.
However, we adopt the convention of explicitly excluding graded Lie algebras of maximal class from the definition of thin Lie algebras.
We also require thin Lie algebras to be infinite-dimensional in this paper.

Thus, a thin Lie algebra must have at least one further diamond past $L_1$, and we let $L_k$ be the earliest (the {\em second} diamond).
It turns out that $k$ can only be one of $3$, $5$, $q$, or $2q-1$, where $q$ is a power of the characteristic $p$ when this is positive
(with only $3$ and $2q-1$ occurring in characteristic two).
This fact was proved in~\cite{AviJur}, but a revised proof is given in~\cite{Mat:chain_lengths}.
In particular, only $3$ and $5$ can occur in characteristic zero.
It is then tempting to call the cases where $k$ is $3$ or $5$ the {\em classical} cases,
following a usage in the theory of (finite-dimensional) simple modular Lie algebras, where the classical ones
are those which are analogues of the simple Lie algebras of characteristic zero
(including the exceptional ones,
as is customary in the modular theory).
The main conclusion of this paper, and its interpretation which we give in Section~\ref{sec:loop},
lends further weight to adopting such terminology.

Thin Lie algebras with $k$ equal to $3$ or $5$, subject to a further restriction $\dim(L_4)=1$ in the former case, and excluding some small characteristics,
were shown in~\cite{CMNS} to belong to (up to) three isomorphism types,
associated to $p$-adic Lie groups of the classical types $A_1$ and $A_2$ (see~\cite{Mat:thin-groups} for matrix realizations of those groups).
In contrast, the values $q$ and $2q-1$ for $k$ occur for two broad classes of thin Lie algebras, of which many were built from certain non-classical finite-dimensional
simple modular Lie algebras, and also to other thin Lie algebras obtained from graded Lie algebras of maximal class through various constructions.
General discussions of those two classes of thin Lie algebras can be found in~\cite{AviMat:A-Z} and~\cite{CaMa:Hamiltonian}, respectively.

Our main result is a general fact that was shown so far in {\em ad hoc} manners for various special instances of the class of thin Lie algebras under consideration.

\begin{theorem}\label{thm:sandwich}
Let $L$ be a thin Lie algebra with second diamond $L_k$, where $k>5$.
Then there is a nonzero element $y$ of $L_1$ such that $[Lyy]=0$.
\end{theorem}

We use the left-normed convention
for iterated Lie products, hence $[abc]$ stands for $[[ab]c]$,
and so $[Lyy]=0$ is another way of saying $(\ad y)^2=0$.

According to a definition of Kostrikin, for $p\neq 2$ the conclusion of Theorem~\ref{thm:sandwich}
says that $y$ is a {\em sandwich element} (or simply a {\em sandwich}).
Kostrikin introduced sandwich elements in the context of the Burnside problem, see~\cite{Kostrikin:Burnside}.
Four decades later they played an important role in the classification theory of finite-dimensional simple modular Lie algebras,
especially for small characteristics (five and seven) in~\cite{PreStr:small_I}.
One remarkable property of sandwich elements, noted by Kostrikin and proved by Premet in~\cite{Premet:degeneration},
is that their presence in a finite-dimensional modular Lie algebra characterizes those simple Lie algebras which are not classical.

Although our thin Lie algebras are infinite-dimensional, a connection with finite-dimensional Lie algebras
comes via a {\em loop algebra} construction, which has been used many times to build thin Lie algebras
from finite-dimensional Lie algebras, simple or close to simple, with respect to certain cyclic gradings.
In Section~\ref{sec:loop} we discuss the relevance of the presence of a sandwich element in this context.

\section{A sandwich element in thin Lie algebras}\label{sec:sandwiches}

A {\em sandwich element} of a Lie algebra $L$
is a nonzero element $c\in L$ such that $(\ad c)^2=0$ and $(\ad c)(\ad z)(\ad c)=0$ for all $z\in L$.
If the characteristic is different from two then the latter requirement (from which the name originates) is superfluous, as it
follows from the former due to
$0=[u[zcc]]=[uzcc]-2[uczc]+[uccz]$
for $u,z\in L$.

In the context of thin Lie algebras it turns out that a nonzero element $y$ of $C_{L_1}(L_2)$
satisfies $[Lyy]=0$ in most cases, as our Theorem~\ref{thm:sandwich} states,
and so $y$ is a sandwich element if the characteristic is not two.
This fact was crucial in the theory of graded Lie algebras of maximal class, making a theory of {\em constituents} possible,
although the sandwich point of view was not explicitly mentioned in~\cite{CMN} or~\cite{CN}.
In particular, \cite[Lemma~3.3]{CMN} can be restated as follows:
if $L$ is a graded Lie algebra of maximal class,
with the centralizer $C_{L_1}(L_2)$ spanned by an element $y$ as customary,
$[L_{i-1}y]=0$ and $[L_iy]\neq 0$ for some $i$, then $[L_{i+1}y]=0$,
whence $[L_iyy]=0$.
A remarkable consequence of this fact is $[Lyy]=0$, which means $(\ad y)^2=0$, and so $y$ is a sandwich element if the characteristic is not two.

Our main goal in this section is showing that such a nonzero element $y$ of $C_{L_1}(L_2)$
is frequently a sandwich element also in a thin Lie algebra $L$, under certain assumptions which we introduce and justify along the way.
If $L$ has second diamond $L_3$, then $C_{L_1}(L_2)=0$, hence we must take $\dim(L_3)=1$ as minimal assumption for our discussion, and then
$\dim\bigl(C_{L_1}(L_2)\bigr)=1$.
Taking then for $y$ a nonzero element of $C_{L_1}(L_2)$ and extending to a basis $x,y$ of $L_1$,
as we do without further mention in this section, we have $[xyy]=0$.

In the case of characteristic two we restrict ourselves to the following simple observation.

\begin{theorem}\label{thm:sandwich_even}
Let $L$ be a thin Lie algebra of characteristic two, with $\dim(L_3)=1$, and
let $y$ span $C_{L_1}(L_2)$.
Then $[Lyy]=0$.
\end{theorem}

\begin{proof}
Because $(\ad y)^2$ is a derivation of $L$, its kernel is a subalgebra.
However, because $[xyy]=0$ both generators
$x$ and $y$ of $L$ belong to the kernel, and hence that is the whole of $L$.
\end{proof}

We will show that Theorem~\ref{thm:sandwich_even} extends to arbitrary characteristic if we include a further hypothesis $\dim(L_5)=1$, as in Theorem~\ref{thm:sandwich}.
A justification for that hypothesis will emerge as we gradually work towards a proof.
We start with suitably extending the conclusion $[L_iyy]=0$ of~\cite[Lemma~3.3]{CMN}, which was about graded Lie algebras of maximal class,
to certain homogeneous components of any thin Lie algebra $L$ with $\dim(L_3)=1$.

\begin{lemma}\label{lemma:uxyy}
Let $L$ be a thin Lie algebra with $\dim(L_3)=1$, and let $y$ span $C_{L_1}(L_2)$.
Let $L_i$ be a homogeneous component with $[L_iy]\neq 0$, and suppose $L_{i-1}$ has a nonzero element $u$ with $[uy]=0$. 
Then $[L_iyy]=0$.
Moreover, $L_i$ and $L_{i+2}$ have dimension one.
\end{lemma}

\begin{proof}
Because of the covering property, $L_i$ is spanned by $[ux]$, and hence has dimension one.
From
\begin{equation}\label{eq:sandwich}
0=[u[xyy]]=[uxyy]-2[uyxy]+[uyyx]=[uxyy]
\end{equation}
it follows that $[L_iyy]=0$.
Because $[uxy]\neq 0$ by hypothesis, the covering property implies that
$L_{i+2}$ is spanned by $[uxyx]$, and hence has dimension one.
\end{proof}

The following immediate consequence of Lemma~\ref{lemma:uxyy} resembles the formulation of~\cite[Lemma~3.3]{CMN} more closely.

\begin{cor}\label{cor:y_centr}
Let $L$ be a thin Lie algebra with $\dim(L_3)=1$, and let $y$ span $C_{L_1}(L_2)$.
If $L_{i-1}$ has a nonzero element centralized by $y$, and $L_i$ has none, then $L_{i+1}$ has such an element as well.
\end{cor}

Equivalently, in a thin Lie algebra with $\dim(L_3)=1$,
at least one of any two consecutive homogeneous components has a nonzero element centralized by $y$.
Corollary~\ref{cor:y_centr} is often useful as simple but weaker replacement for Lemma~\ref{lemma:uxyy}.
However, we need the full strength of Lemma~\ref{lemma:uxyy} in order to deduce the following consequence.

\begin{cor}\label{cor:vyy}
Let $L$ be a thin Lie algebra with $\dim(L_3)=1$,
and let $y$ span $C_{L_1}(L_2)$.
Suppose $\dim(L_j)=2$ for some $j>1$.
Then $\dim(L_{j-1})=1$ and $[L_{j-1}yy]=0$.
In particular, $L_j$ contains a nonzero element centralized by $y$.
\end{cor}

\begin{proof}
Because of the covering property $L_{j-1}$ cannot contain any nonzero element centralized by $y$.
Hence $L_{j-2}$ does contain a nonzero element $u$ with $[uy]=0$.
Consequently, $v=[ux]$ spans $L_{j-1}$,
and $[vyy]=0$ according to Lemma~\ref{lemma:uxyy}.
Because $\dim(L_j)=2$ the element $[vy]$ of $L_j$ is nonzero, and is centralized by $y$ as desired.
\end{proof}

In particular, Corollary~\ref{cor:vyy} implies that two consecutive components in a thin Lie algebra $L$
with $\dim(L_3)=1$ cannot both be diamonds.
Thus, the {\em distance} between any two consecutive diamonds, by which we mean the difference of their degrees, is at least two.

An easily proved consequence of this fact is that $L$ remains thin under base field extensions.
By contrast, this property fails for a thin Lie algebra $L$ with two diamonds occurring as consecutive components, as we show in Remark~\ref{rem:field_extensions}.
In fact, the thin Lie algebras studied in~\cite{GMY,ACGMNO}
where all homogeneous components except $L_2$ are diamonds, require their base field to have a quadratic field extension.

\begin{rem}\label{rem:field_extensions}
Although it is inconsequential for this paper, we explain why in a thin Lie algebra $L$
over an algebraically closed field no two diamonds can occur as consecutive components.
For a graded Lie algebra $L=\bigoplus_{i=1}^{\infty}L_i$ we consider the linear maps
$\psi_j:L_j\to\Hom_F(L_1,L_{j+1})$ obtained by restriction from the adjoint representation.
When $L$ is thin with $L_j$ and $L_{j+1}$ both diamonds,
the covering property for the homogeneous component $L_j$ means that $\psi(L_j)$ is a two-dimensional subspace of $\Hom_F(L_1,L_{j+1})$
where every nonzero element is a surjective linear map.
Upon composing with a linear bijection $L_{j+1}\to L_1$ (whose choice is immaterial) we may view that as a two-dimensional subspace of
$\Hom_F(L_1,L_1)$ where every nonzero element has nonzero determinant.
In terms of the associated projective spaces, $P\bigl(\psi(L_j)\bigr)$ is then
a one-dimensional projective subspace of the three-dimensional projective space $P\bigl(\Hom_F(L_1,L_1)\bigr)$,
disjoint from the non-degenerate quadric given by the zeroes of the determinant map.
That cannot happen if the base field $F$ is algebraically closed.
\end{rem}

Corollary~\ref{cor:vyy} is relevant in assigning {\em types} to certain diamonds of a thin Lie algebra.
This was done in different ways according to whether the second diamond
occurs in degree $q$ or $2q-1$, in~\cite{CaMa:Nottingham} and~\cite{CaMa:thin}, respectively.
However, in both cases necessary conditions for a diamond $L_j$ to be assigned a type
were $\dim(L_{j-1})=1$ and $[L_{j-1}yy]=0$.
Corollary~\ref{cor:vyy} shows that those conditions hold in any thin Lie algebra with $\dim(L_3)=1$.

Now we use Lemma~\ref{lemma:uxyy} to prove that if $L$ is thin with $\dim(L_3)=1$, then
$[L_iyy]=0$ for all one-dimensional components $L_i$.

\begin{lemma}\label{lemma:sandwich}
Let $L$ be a thin Lie algebra with $\dim(L_3)=1$,
and let $y$ span $C_{L_1}(L_2)$.
Suppose $[L_jyy]\neq 0$ for some $j$.
Then $\dim(L_j)=2$. 
\end{lemma}

\begin{proof}
Note that $j>2$.
If $L_{j-1}$ had a nonzero element centralized by $y$, then because $[L_jy]\neq 0$ by hypothesis,
Lemma~\ref{lemma:uxyy} would apply with $i=j$ and yield $[L_jyy]=0$, a contradiction.
Therefore, $L_{j-1}$ does not have any nonzero element centralized by $y$.
Then because of Corollary~\ref{cor:y_centr} each of $L_{j-2}$ and $L_{j}$ has such an element.
In particular, because $[L_jy]\neq 0$ we must have $\dim(L_j)=2$.
\end{proof}

Thus, the task of proving $[Lyy]=0$ in a thin Lie algebra $L$ has now been reduced to showing $[L_jyy]$ for the diamonds $L_j$.
We pause to note that $L$ having any two diamonds at distance two implies $[Lyy]\neq 0$.
In fact, if $L_j$ and $L_{j+2}$ are diamonds, in a thin Lie algebra $L$ with $\dim(L_3)=1$,
then the covering property easily implies $[L_jy]\neq 0$, whence $[L_jy]=L_{j+1}$
because $L_{j+1}$ is one-dimensional.
Consequently, $[L_jyy]=[L_{j+1}y]\neq 0$, because $[L_{j+1}y]$ and $[L_{j+1}y]$ span the diamond $L_{j+2}$.

The following remark describes one special thin Lie algebra with $\dim(L_3)=1$
having certain diamonds at distance two.

\begin{rem}\label{rem:diamond_deg_5}
The proof of~\cite[Theorem~2(a)]{CMNS} implies that in characteristic either zero or larger than five there is a unique (infinite-dimensional)
thin Lie algebra $L$ having second diamond $L_5$.
That result, which predates the study of thin Lie algebras {\em per se,}
only claims uniqueness for the graded Lie algebra $L$ associated to the lower central series
of a infinite thin pro-$p$ group with second diamond in weight $5$, for $p>5$.
However, its proof applies to any thin Lie algebra $L$ with second diamond $L_5$, for $p>5$ or $p=0$.
A construction for that thin Lie algebra $L$ was also given in~\cite{CMNS}, and that is valid in every characteristic except for characteristic two.
The diamonds of $L$ occur in each degree congruent to $\pm 1$ modulo $6$, and so as noted above one has
$[L_{i-1}yy]\neq 0$ for each $i$ multiple of six.
Uniqueness of $L$ fails in characteristic three because $5=2q-1$ with $q=3$, and in characteristic five because we may have $5=q$ then,
and so lots of other thin Lie algebras enter those cases (see Remark~\ref{rem:char_5}).
\end{rem}

Our next remark describes how in characteristic three one can produce uncountably many thin Lie algebras with $\dim(L_3)=1$
having diamonds at distance two.

\begin{rem}\label{rem:char_3}
According to~\cite[Section~9]{CMN}, for each power $q$ of the (positive) characteristic
there are uncountably many infinite-dimensional graded Lie algebras of maximal class $M$
with precisely two distinct two-step centralizers
and constituent sequence beginning with $2q,q,q$.
As shown in ~\cite[Section~5]{CaMa:thin},
each such $M$ has a maximal subalgebra $L$ which becomes thin under a new grading, and its diamonds start with $L_1,L_{2q-1},L_{3q-2},L_{4q-3}$.
In characteristic three and taking $q=3$ those diamonds are $L_1,L_5,L_7,L_9$.
Consequently, $[L_5yy]\neq 0$ and $[L_7yy]\neq 0$.
\end{rem}

The thin Lie algebras of the above remarks suggest that we might be able to obtain more information
on the earliest diamond $L_j$ with $[L_jyy]\neq 0$
than on an arbitrary one,
hence we will do that in preparation for a proof of Theorem~\ref{thm:sandwich_odd}.

\begin{lemma}\label{lemma:sandwich_plus}
Let $L$ be a thin Lie algebra with $\dim(L_3)=1$,
and let $y$ span $C_{L_1}(L_2)$.
Suppose $[Lyy]\neq 0$, and let $j$ be minimal such $[L_jyy]\neq 0$.
Then $L_{j-2}$ and $L_{j-3}$ are one-dimensional and centralized by $y$.
\end{lemma}

\begin{proof}
According to Lemma~\ref{lemma:sandwich} we have $\dim(L_j)=2$,
and hence $\dim(L_{j-1})=1$ because of Corollary~\ref{cor:vyy}.
Note that $j>3$.
Let $t$ be any nonzero element of $L_{j-3}$.
Then $[tyy]=0$ and $[txyy]=0$ by minimality of $j$, and hence
\[
0=[t[xyy]]=[txyy]-2[tyxy]+[tyyx]=-2[tyxy].
\]
Because of Theorem~\ref{thm:sandwich_even} we are not in characteristic two, and so we conclude $[tyxy]=0$.
But then $[tyx]=0$, because $L_{j-1}$ contains no nonzero element centralized by $y$.
The covering property now implies $[ty]=0$, otherwise
$[tyx]$ and $[tyy]$, which both vanish, would have to span $L_{j-1}$.
Because $t$ was an arbitrary nonzero element of $L_{j-3}$,
Corollary~\ref{cor:vyy} implies $\dim(L_{j-3})=1$.
Finally, $L_{j-2}$ has dimension one because it is spanned by $[tx]$.
But we know that $L_{j-2}$ contains a nonzero element centralized by $y$, and hence $[L_{j-2}y]=0$.
\end{proof}

We are ready to prove the harder analogue of Theorem~\ref{thm:sandwich_even} for odd characteristics, thus completing a proof of Theorem~\ref{thm:sandwich}.
As our discussion leading to Remark~\ref{rem:diamond_deg_5} shows, we need an additional assumption to ensure that
no diamonds occur at distance two.
As it turns out, assuming $\dim(L_5)=1$ will do, which is another way of asking that the second diamond of $L$ occurs past $L_5$.

\begin{theorem}\label{thm:sandwich_odd}
Let $L$ be a thin Lie algebra  of odd characteristic, with
$\dim(L_3)=\dim(L_5)=1$, and
let $y$ span $C_{L_1}(L_2)$.
Then $[Lyy]=0$.
\end{theorem}

\begin{proof}
Besides the relation $[yxy]=0$, which serves to define $y$ up to a scalar, in $L$ we also have $[yxxy]=0$ and $[yxxxy]=0$.
In fact, the former assertion follows from $0=[yx[yx]]=[yxyx]-[yxxy]=[yxxy]$,
and hence $[yxxx]$ spans $L_4$.
To prove the latter, assuming $[yxxxy]\neq 0$ for a contradiction, it will span $L_5$ because of our hypothesis $\dim(L_5)=1$.
Because $0=[yxx[xyy]]=[yxxxyy]$ we then have $[L_5y]=0$, and hence $[yxxxyx]$ spans $L_6$.
Now the generalized Jacobi identity yields the contradiction
$0=[y[yxxxx]]=-4[yxxxyx]$.

Now suppose for a contradiction that $[Lyy]\neq 0$, and let $j$ be minimal such that $[L_jyy]\neq 0$.
We will use the relations $[yxxy]=0$ and $[yxxxy]=0$ to derive a contradiction.
Let $t$ be a nonzero element of $L_{j-3}$.
According to Lemma~\ref{lemma:sandwich} we have $[ty]=0$ and $[txy]=0$,
and $L_{j-1}$ is spanned by $v:=[txx]$.
Furthermore, $\dim(L_j)=2$, and hence $[vyy]=0$ according to Corollary~\ref{cor:vyy}.
The calculations
\[
0=[v[xyy]]=[vxyy]-2[vyxy],
\]
and
\[
0=[t[yxxxy]]=[t[yxxx]y]=-3[txxyxy]+[txxxyy]=[vxyy]-3[vyxy],
\]
taken together imply $[vxyy]=[vyxy]=0$.
Because $[vx]$ and $[vy]$ span $L_j$ we obtain $[L_jyy]=0$,
which gives the desired contradiction.
\end{proof}

We discuss the extent to which hypothesis $\dim(L_5)=1$ of Theorem~\ref{thm:sandwich_odd} is necessary.
As we recalled in Remark~\ref{rem:diamond_deg_5},
over a field of characteristic $p>5$ that hypothesis serves to exclude a single exception,
given by the unique thin Lie algebra with second diamond $L_5$, identified in~\cite[Theorem~2(a)]{CMNS} and described in~Remark~\ref{rem:diamond_deg_5}.
However, if we omit the hypothesis $\dim(L_5)=1$ of Theorem~\ref{thm:sandwich_odd} in characteristic $p=3$,
its conclusion $[Lyy]=0$ is violated by uncountably many thin Lie algebras with second diamond $L_5$ and third diamond $L_7$,
which we described in Remark~\ref{rem:char_3}.
The following remark focuses on the borderline case $p=5$.

\begin{rem}\label{rem:char_5}
We discuss how far Theorem~\ref{thm:sandwich_odd} may be extended to include
thin Lie algebras of characteristic five, with second diamond $L_5$ and an additional assumption.
Countably many thin Lie algebras with second diamond $L_q$ were constructed in~\cite{AviMat:A-Z},
for any power $q$ of the odd characteristic.
In particular, when $q=5$ this gives us a countable family of thin Lie algebras with second diamond $L_5$.
However, all those thin Lie algebras, which have third diamond $L_{2q-1}=L_9$
(possibly {\em fake,} see~\cite{AviMat:A-Z} for what that means), do satisfy the conclusion $[Lyy]=0$ of Theorem~\ref{thm:sandwich_odd}.

In fact, one can prove that $[Lyy]=0$ holds for any thin Lie algebra $L$ of characteristic five satisfying
$\dim(L_3)=1$, $\dim(L_5)=2$, and $[L_7y]=0$.
(Here the hypothesis $[L_7y]=0$ implies that neither $L_7$ nor $L_8$ is a diamond, and hence the third diamond of $L$ does not occur earlier than $L_9$.)
A proof of this fact follows the general inductive strategy employed in~\cite{AviMat:Nottingham_structure},
but is not a formal consequence of~\cite{AviMat:Nottingham_structure},
where some generality was sacrificed in favour of simpler exposition.
In fact, part of that simplification in~\cite{AviMat:Nottingham_structure} relies on using our Theorem~\ref{thm:sandwich_odd}.
\end{rem}

\section{Thin loop algebras and the sandwich element}\label{sec:loop}
We conclude this paper with a discussion of the significance for thin Lie algebras of an important property of sandwich elements.
Classical simple Lie algebras
do not have sandwich elements, see~\cite[p.~124]{Sel}.
By contrast, confirming a conjecture of Kostrikin, Premet proved in~\cite{Premet:degeneration} that
every finite-dimensional simple Lie algebra, over an algebraically closed field of characteristic $p>5$,
which is not classical, must have sandwich elements (that is, have {\em strong degeneration}).

The connection of this characterization of classical Lie algebras with (infinite-dimensional) thin Lie algebras
comes from the fact that several of the latter have been constructed as {\em loop algebras}
of certain finite-dimensional simple Lie algebras, or close to simple.
In the simplest setting, one starts from a finite-dimensional simple Lie algebra $S$ over a field $F$,
with a cyclic grading $S=\bigoplus_{k\in\Z/N\Z} S_k$, and considers the Lie algebra $S\otimes F[t]$ over $F$, where $t$ is an indeterminate.
Its subalgebra
$L=\bigoplus_{k>0}S_{\bar k}\otimes t^k$,
where $\bar k=k+N\Z$,
is naturally graded over the positive integers, and is called a {\em loop algebra} in this context.
In certain cases one needs a slightly more general construction involving also a derivation of $S$
(such as that of~\cite[Definition~2.1]{AviMat:A-Z}),
but that is inconsequential for our present observation.

It was proved in~\cite{CMNS} that a thin Lie algebra $L$ (of characteristic not $2$ or $3$)
having second diamond $L_3$ and $\dim(L_4)=1$ belongs to one of (up to) two isomorphism types,
and a thin Lie algebra $L$ (of characteristic not $2$, $3$ or $5$) with second diamond $L_5$ is unique up to isomorphism
(see also Remark~\ref{rem:diamond_deg_5}).
Each of those Lie algebras can be realized as a loop algebra of a classical simple Lie algebra of type $A_1$ or $A_2$.

Now, the fact that those thin Lie algebras of~\cite{CMNS} are loop algebras of classical simple Lie algebras implies
that there cannot be any nonzero element $y\in L_{\bar 1}$ with $(\ad y)^2=0$, because such $y$ would have the form
$y=c\otimes t$ for some sandwich element $c$ of $S$, which cannot exist.
For the same reason, if a thin Lie algebra $L$ (say of characteristic $p>5$) with second diamond past $L_5$ is a loop algebra of a simple Lie algebra $S$,
then $S$ cannot be classical, because $L$ has a sandwich element $y$ according to Theorem~\ref{thm:sandwich_odd}, and hence so does $S$.

In fact, all constructions of thin Lie algebras with second diamond $q$ or $2q-1$ as loop algebras which were given
in several papers~\cite{Car:Nottingham,AviMat:-1,CaMa:Hamiltonian,AviMat:A-Z,AviMat:mixed_types}
start from non-classical simple Lie algebras of various types.
The occurrence of non-classical algebras in those constructions is hardly surprising,
because the expected structure of such thin Lie algebras $L$ given by corresponding uniqueness results
shows that the dimension of $S$ ought to be a power of $p$ or one or two less in the various cases.
However, the sandwich element $y$ does provide guidance in identifying the appropriate cyclic grading of $S$ employed in those constructions.

\bibliography{References}

\end{document}